\documentclass[english,11pt]{smfart}

\usepackage{etex}

\usepackage{amsbsy}
\usepackage{amsmath,amsfonts,amssymb,amsthm,mathrsfs,mathtools}

\usepackage{bm}

\usepackage[a4paper,vmargin={3cm,3cm},hmargin={3cm,3cm}]{geometry}
\usepackage[font=sf, labelfont={sf,bf}, margin=1cm]{caption}
\usepackage{graphicx}
\usepackage{epsfig}
\usepackage{latexsym}
\usepackage{xcolor}
\usepackage{ae,aecompl}
\usepackage{soul,framed}
\usepackage{comment}

\usepackage{xcolor}
\usepackage[pdfpagemode=UseNone,bookmarksopen=false,colorlinks=true,urlcolor=blue,citecolor=blue,citebordercolor=blue,linkcolor=blue]{hyperref}
\usepackage{smfhyperref}
\usepackage[capitalize]{cleveref}

\usepackage{pstricks}
\usepackage{enumerate}
\usepackage{tikz,animate,media9}						
\usepackage{todonotes}
\usepackage{pifont}
\usepackage{bm,marvosym}
\usepackage{algorithm}
\usepackage{algorithmic}

\usepackage{bbm}

\usepackage{tcolorbox}

\definecolor{aleacolor}{rgb}{0.16,0.59,0.78}

\hypersetup{
	breaklinks,
	colorlinks=true,
	linkcolor=aleacolor,
	urlcolor=aleacolor,
	citecolor=aleacolor}

\newcount\colveccount
\newcommand*\colvec[1]{
	\global\colveccount#1
	\begin{pmatrix}
		\colvecnext
	}
	\def\colvecnext#1{
		#1
		\global\advance\colveccount-1
		\ifnum\colveccount>0
		\\
		\expandafter\colvecnext
		\else
	\end{pmatrix}
	\fi
}

\newcommand{\ndR}{\mathbb{R}}

\renewcommand{\Pr}[1]{\mathbb{P}(#1)}

\newcommand{\convdis}{\,{\buildrel d \over \longrightarrow}\,}

\newcommand{\convas}{\,{\buildrel a.s. \over \longrightarrow}\,}

\newcommand{\Di}{\mathrm{D}}

\newcommand{\mX}{\mathsf{X}}
\newcommand{\mY}{\mathsf{Y}}

\newtheorem{theorem}{Theorem}[section]

\newtheorem{corollary}[theorem]{Corollary}
\newtheorem{proposition}[theorem]{Proposition}
\newtheorem{lemma}[theorem]{Lemma}

\numberwithin{equation}{section}

\title[Mass and radius of balls in GHP convergent sequences]{\textbf{Mass and radius of balls in Gromov--Hausdorff--Prokhorov convergent sequences}}
\date{}

\author{Benedikt Stufler}

\address[Benedikt Stufler]{Vienna University of Technology}
\email{benedikt.stufler [at] tuwien.ac.at}

\begin{document}

\begin{abstract}
We survey some properties of Gromov--Hausdorff--Prokhorov convergent sequences $(\mathsf{X}_n, d_{\mathsf{X}_n}, \nu_{\mathsf{X}_n})_{n \ge 1}$ of random compact metric spaces equipped with Borel probability measures. We formalize that if the limit is almost surely non-atomic, then for large $n$ each open ball in $\mathsf{X}_n$ with small radius must have small mass. Conversely, if the limit is almost surely fully supported, then each closed ball in $\mathsf{X}_n$ with small mass must have small radius.

We do not claim any new results, but justifications are provided for properties for which we could not find explicit references.
\end{abstract}


\maketitle

\section{Introduction}

The Gromov--Hausdorff--Prokhorov  distance $d_{\mathrm{GHP}}$ is a metric on the collection of  $\mathfrak{K}$ of (representatives of equivalence classes of) compact metric spaces endowed with Borel probability measures. Given $(X, d_X, \mu_X), (Y, d_Y, \mu_Y) \in \mathfrak{K}$, it may be defined by
\[
d_{\textsc{GHP}}( (X, d_X, \mu_X), (Y, d_Y, \mu_Y)) = \inf_{E, \varphi_X, \varphi_Y} \max(d_{\textsc{H}}(\varphi_X(X), \varphi_Y(Y)), d_{\textsc{P}}(\mu_X\varphi_X^{-1}, \mu_Y\varphi_Y^{-1} ) )
\]
with the index ranging over all isometric embeddings $\varphi_X: X \to E$ and $\varphi_Y: Y \to E$ into any common metric space $E$. Here $d_{\textsc{H}}(\varphi_X(X), \varphi_Y(Y))$ denotes the Hausdorff distance of the images of $X$ and $Y$, and $d_{\textsc{P}}(\mu_X\varphi_X^{-1}, \mu_Y\varphi_Y^{-1} )$ denotes the Prokhorov distances of the push-forwards of the measures $\mu_X$ and $\mu_Y$ along $\varphi_X$ and $\varphi_Y$. Detailed expositions of this concept can be found in \cite[Ch. 7]{MR1835418},~\cite[Ch. 27]{zbMATH05306371},~\cite{zbMATH06247183} \cite[Sec. 6]{MR2571957}, and the recent survey~\cite{janson2020gromovprohorov}.

In the following sections we survey selected properties on the mass of $\epsilon$-balls in convergent sequences of random elements $(\mathsf{X}_n, d_{\mathsf{X}_n}, \nu_{\mathsf{X}_n})_{n \ge 1}$ of $\mathfrak{K}$. Specifically, we  formalize that if the limit is almost surely non-atomic, then for large $n$ each open $\epsilon$-ball in $\mathsf{X}_n$ with small $\epsilon$ must have small mass. Conversely, if the limit is almost surely fully supported, then each closed $\epsilon$-ball in $\mathsf{X}_n$ with small mass must have small radius $\epsilon$.

We do not claim any new results, but justifications are provided for statements for which we could not find explicit references.



\section{The largest mass of open $\epsilon$-balls}

For any $\epsilon>0$ and any element $(X, d_X, \nu_X) \in \mathfrak{K}$ we define
\begin{align}
	\label{eq:deflambda}
	\lambda_\epsilon^X := \sup_{x \in X} \nu_X(B_\epsilon^X(x)).
\end{align}
Here 
\begin{align}
	B_\epsilon^X(x) = \{y \in X \mid d_X(x,y) < \epsilon\}
\end{align} denotes the open $\epsilon$-ball in $X$.

\begin{lemma}
	\label{le:lambda}
	Let $(X, d_X, \nu_X), (Y, d_Y, \nu_Y) \in \mathfrak{K}$ satisfy 
	\[
	d_{\mathrm{GHP}}\left( (X, d_X, \nu_X), (Y, d_Y, \nu_Y) \right) < \epsilon'
	\]
	for some $\epsilon'>0$. Then  for all $\epsilon>0$
	\[
	\lambda_\epsilon^X \le \lambda_{\epsilon + 2 \epsilon'}^Y + \epsilon'.
	\]
\end{lemma}
\begin{proof}
	It follows from the assumed bound on the Gromov--Hausdorff--Prokhorov distance that there is a metric space $(Z, d_Z)$ and an isometric embeddings $\varphi:~X \to Z$ and $\psi: Y \to Z$ such that the Hausdorff distance between the images $\varphi(X)$ and $\psi(Y)$  is smaller than $\epsilon'$, and the Prokhorov distance between the push-forward measures $\nu_X^\varphi$ and $\nu_Y^\psi$ is also smaller than $\epsilon'$. Let $x$ be an arbitrary point of $X$. Using $d_{\mathrm{H}}(\varphi(X), \psi(Y)) < \epsilon'$, it follows that there is a point $y$ of $Y$ with
	\[
	d_{Z}(\varphi(x), \psi(y)) < \epsilon'.
	\]
	Clearly this implies for any $\epsilon>0$
	\[
	B_\epsilon^Z(\varphi(x)) \subset B_{\epsilon + \epsilon'}^Z(\psi(y)).
	\]
	Hence
	\[
	\nu_X^\varphi(B_\epsilon^Z(\varphi(x))) \le \nu_X^\varphi(B_{\epsilon + \epsilon'}^Z(\psi(y))).
	\]
	Using $d_{\mathrm{P}}(\nu_X^\varphi, \nu_Y^\psi) < \epsilon'$, it follows that
	\[
	\nu_X^\varphi(B_{\epsilon + \epsilon'}^Z(\psi(y))) \le \nu_Y^\psi(B_{\epsilon + 2\epsilon'}^Z(\psi(y))) + \epsilon'.
	\]
	Combining the last two inequalities and using that $\varphi, \psi$ are isometries, it follows that
	\[
	\nu_X(B_{\epsilon}^X(x)) \le \nu_Y(B_{\epsilon + 2\epsilon'}^Y(y)) + \epsilon'.
	\] 
	As this holds for an arbitrary $x \in X$, it follows that
	\[
	\lambda_\epsilon^X \le \lambda_{\epsilon + 2\epsilon'}^Y  + \epsilon'.
	\]
\end{proof}

\begin{corollary}
	\label{co:uppercontlambda}
	\begin{enumerate}
		\item Let $(X, d_X, \nu_X) \in \mathfrak{K}$. The function \[ [0,\infty[ \to \ndR, \epsilon \mapsto \lambda_\epsilon^X\] increases monotonically and is left-continuous.
		\item 	Fix $\epsilon>0$. The function \[\mathfrak{K} \to \ndR, (X, d_X, \nu_X) \mapsto \lambda_\epsilon^X\] is lower semi-continuous.
	\end{enumerate}
\end{corollary}

\begin{proof}
	We start with the first claim.  
	It is clear that the function $\lambda_{(\cdot)}^X$  increases monotonically. Suppose that there is a point $\epsilon\ge 0$ where it is not left-continuous. As it is monotonically increasing it has one-sided limits at each point,  entailing that $\lim_{t \uparrow \epsilon} \lambda_t^X - \lambda_\epsilon^X >0$. Hence there is a point $x \in X$ and a $\delta>0$ such that \[
	\nu_X(B_{\epsilon - 1/n} (x)) + \delta \le  \nu_X(B_\epsilon(x))\] for all $n \ge 1$. Thus
	\[
	\nu_X(B_{\epsilon - 1/n} (x) \setminus B_\epsilon(x)) \ge \delta
	\]
	for all $n$. But the left-hand side tends to zero as $n$ becomes large. Hence we have proven by contradiction, that $\lambda_{(\cdot)}^X$ is left-continuous on $[0,\infty[$. This completes the proof of the first claim.
	
	As for the second claim, suppose that  $(X,d_X, \nu_X) \in \mathfrak{K}$ is the Gromov--Hausdorff--Prokhorov limit of a sequence $(X_n, d_{X_n}, \nu_{X_n})_{n \ge 1}$ in $\mathfrak{K}$.  Lemma~\ref{le:lambda} implies that for some sequence $t_n>0$ with $t_n=o(1)$
	\[
	\lambda_{\epsilon - 2 t_n}^{X} + t_n \le \lambda_{\epsilon}^{X_n}.
	\]
	As $\lambda_{(\cdot)}^{X}$ is left-continuous, it follows that 
	\[
	\lambda_{\epsilon}^{X_n} \ge \lambda_{\epsilon}^X + o(1).
	\]
	Hence $\lambda_{\epsilon}^{(\cdot)}$ is lower semi-continuous.
\end{proof}

\begin{corollary}
	\label{co:fulllambda}
	Suppose that $(X,d_X, \nu_X)$ is the Gromov--Hausdorff--Prokhorov limit of a sequence $(X_n, d_{X_n}, \nu_{X_n})_{n \ge 1}$.  The following statements are equivalent:
	\begin{enumerate}
		\item There is no point $x \in X$ with $\nu_X(\{x\}) >0$.
		\item $\lim_{\epsilon \downarrow 0} \lambda_\epsilon^X=0$.
		\item $\lim_{\epsilon \downarrow 0} \limsup_{n \to \infty} \lambda_\epsilon^{X_n}=0$
	\end{enumerate}
\end{corollary}
\begin{proof}
	The equivalence of the first two statements follows from the compactness of $X$. Indeed, the existence of a  singleton with positive mass $c>0$ would imply $\lim_{\epsilon \downarrow 0} \lambda_\epsilon^X\ge c $. Conversely, $\lim_{\epsilon \downarrow 0} \lambda_\epsilon^X>0$ implies the existence of a lower bound $c>0$ and for each $n \ge 1$ a point $x_n \in X$ and a radius $\epsilon_n>0$ such that $\nu_X(B_{\epsilon_n}(x_n)) > c$ for all $n \ge 1$ and $\lim_{n \to \infty} \epsilon_n= 0$. Since $X$ is compact, it follows that $x_n$ admits a limit point $x \in X$ as $n$ tends to infinity along some subsequence. This entails $\nu_X(B_\epsilon(x))\ge c$ for any $\epsilon>0$, and hence $\nu_X(\{x\}) \ge c>0$.

	The equivalence of the second and third statement follow from the inequality
	\[
	\lambda_\epsilon^{X_n} \le \lambda_{2 \epsilon}^X + o(1) \le \lambda_{3 \epsilon}^{X_n} + o(1)
	\]
	which holds for each $\epsilon>0$ by Lemma~\ref{le:lambda}.
\end{proof}

The following version for random elements of~$\mathfrak{K}$ will be used later on.

\begin{corollary}
	\label{co:conconlambda}
	Let $(\mX,d_{\mX}, \nu_{\mX})$ and $(\mX_n, d_{\mX_n}, \nu_{\mX_n})_{n \ge 1}$ be random elements of $\mathfrak{K}$ satisfying
	\[
	(\mX_n, d_{\mX_n}, \nu_{\mX_n}) \convdis (\mX,d_{\mX}, \nu_{\mX}).
	\]
	Then the following statements are equivalent:
	\begin{enumerate}
		\item Almost surely there is no $x \in \mX$ with $\nu_\mX(\{x\})>0$.
		\item For all $\epsilon, \epsilon'>0$ there exist $\delta>0$ and $N>0$ such that for all $n \ge N$: $\Pr{\lambda_\delta^{\mX_n} > \epsilon} < \epsilon'$.
	\end{enumerate}
\end{corollary}
\begin{proof}
	By Corollary~\ref{co:uppercontlambda} it holds for all $\epsilon>0$ that  the map $\lambda_{\epsilon}^{(\cdot)}$  is lower semi-continuous with respect to the Gromov--Hausdorff--Prokhorov metric. Thus, for any $\epsilon, \delta>0$, the subset $ \{ (Z, d_Z, \nu_Z) \in \mathfrak{K} \mid \lambda_\delta^Z > \epsilon\}$ is open. By the Portmanteau theorem
	\begin{align}
		\label{eq:dohlambda}
		\liminf_{n \to \infty} \Pr{ \lambda_\delta^{\mX_n} > \epsilon} \ge \Pr{ \lambda_\delta^{\mX} > \epsilon} .
	\end{align}
	
	Now suppose that with a positive probability there is a point $x \in \mX$ with $\nu_{\mX}(\{x\}) >0$. Then there exist $\epsilon, \epsilon'>0$ with $\Pr{\lambda_\delta^{\mX} > \epsilon} \ge \epsilon'$ for all $\delta>0$. It follows by \eqref{eq:dohlambda}  that $\Pr{ \lambda_\delta^{\mX_n} > \epsilon} \ge \epsilon'$ for all sufficiently large $n$. As this holds for each $\delta>0$, we have shown by contraposition that Claim 2 implies Claim 1. 
	
	It remains to show that Claim 1 implies Claim 2. That is, we assume now that almost surely  all $x \in \mX$ satisfy $\nu_{\mX}(\{x\}) = 0$. By Skorokhod's representation theorem we may assume that $\mX_n$ converges almost surely towards $\mX$. It follows from Corollary~\ref{co:fulllambda}  that  \[
	\lim_{\delta \downarrow 0} \limsup_{n \to \infty} \lambda_\delta^{\mX_n}=0
	\] holds almost surely.  Consequently, for any $\epsilon', \epsilon>0$ we may select $\delta>0$ and  such that \[\Pr{\limsup_{n \to \infty} \lambda_\delta^{\mX_n} > \epsilon/2} < \epsilon'.\]
	Hence there exists $N \ge 1$ such that for all $n \ge N$
	\[
	\Pr{\lambda_{\delta}^{\mX_n} > \epsilon/2 } < \epsilon'.
	\]
\end{proof}

\section{The minimal mass of closed $\epsilon$-balls}
For any $\epsilon>0$ and any element $(X, d_X, \nu_X) \in \mathfrak{K}$ we define
\begin{align}
	\label{eq:def}
	\rho_\epsilon^X := \inf_{x \in X} \nu_X(C_\epsilon^X(x)).
\end{align}
Here 
\begin{align}
	C_\epsilon^X(x) = \{y \in X \mid d_X(x,y) \le \epsilon\}
\end{align} denotes the closed $\epsilon$-ball in $X$. 

\begin{lemma}
	\label{le:rho}
	Let $(X, d_X, \nu_X), (Y, d_Y, \nu_Y) \in \mathfrak{K}$ satisfy 
	\[
	d_{\mathrm{GHP}}\left( (X, d_X, \nu_X), (Y, d_Y, \nu_Y) \right) < \epsilon'
	\]
	for some $\epsilon'>0$. Then it holds for all $\epsilon>0$ that
	\[
	\rho_\epsilon^X \le \rho_{\epsilon + 2 \epsilon'}^Y + \epsilon'.
	\]
\end{lemma}
\begin{proof}
	By assumption, there is a metric space $(Z, d_Z)$ and an isometric embeddings $\varphi:~X \to Z$ and $\psi: Y \to Z$ such that the Hausdorff distance between the images $\varphi(X)$ and $\psi(Y)$  is smaller than $\epsilon'$, and the Prokhorov distance between the push-forward measures $\nu_X^\varphi$ and $\nu_Y^\psi$ is also smaller than $\epsilon'$. Let $y \in Y$ be given. Using the bound for the Hausdorff distance, it follows that there is a point $x \in X$ with
	\[
	d_{Z}(\varphi(x), \psi(y)) < \epsilon'.
	\]
	Using the bound for the Prokhorov distance, it follows that
	\begin{align*}
		\rho_\epsilon^X &\le \nu_X(C_\epsilon^X(x)) \\
		&\le \nu_X^\varphi( C_\epsilon^Z(\varphi(x))) \\
		&\le \nu_X^\varphi( C_{\epsilon+ \epsilon'}^Z(\psi(y))) \\
		&\le \nu_Y^\psi( C_{\epsilon+ 2\epsilon'}^Z(\psi(y))) + \epsilon'.
	\end{align*} 
	Since $\psi$ is an isometry, it follows that
	\[
	\psi^{-1}(C_{\epsilon+ 2\epsilon'}^Z(\psi(y))) =  C_{\epsilon+ 2\epsilon'}^Y(y).
	\]
	Hence
	\[
	\rho_\epsilon^X \le \nu_Y( C_{\epsilon+ 2\epsilon'}^Y(y)) + \epsilon'.
	\]
	As this holds for all $y \in Y$, it follows that
	\[
	\rho_\epsilon^X - \epsilon' \le \rho_{\epsilon + 2 \epsilon'}^Y.
	\]
\end{proof}

\begin{corollary}
	\label{co:uppercont}
	\begin{enumerate}
		\item Fix $(X, d_X, \nu_X) \in \mathfrak{K}$. The function \[ [0,\infty[ \to \ndR, \epsilon \mapsto \rho_\epsilon^X\] increases monotonically and is right-continuous.
		\item 	Fix $\epsilon>0$. The function \[\mathfrak{K} \to \ndR, (X, d_X, \nu_X) \mapsto \rho_\epsilon^X\] is upper semi-continuous.
	\end{enumerate}
\end{corollary}
Compare with~\cite[Lem. 3.2]{zbMATH06610054}.

\begin{proof}
	As for the first claim, it is clear that $\rho_{(\cdot)}^X$  increases monotonically. Suppose that there is a point $\epsilon\ge 0$ where it is not right-continuous. As it is monotonically increasing it has one-sided limits at each point,  entailing that $\lim_{t \downarrow \epsilon} \rho_t^X - \rho_\epsilon^X >0$. It follows that there is a point $x \in X$ and a $\delta>0$ such that \[
	\nu_X(C_{\epsilon + 1/n} (x)) \ge \delta + \nu_X(C_\epsilon(x))\] for all $n \ge 1$. Thus
	\[
	\nu_X(C_{\epsilon + 1/n} (x) \setminus C_\epsilon(x)) \ge \delta
	\]
	for all $n$. But the left-hand side tends to zero as $n$ becomes large. Hence we have proven by contradiction, that $\rho_{(\cdot)}^X$ is right-continuous on $[0,\infty[$. This completes the proof of the first claim.
	
	As for the second claim, suppose that  $(X,d_X, \nu_X) \in \mathfrak{K}$ is the Gromov--Hausdorff--Prokhorov limit of a sequence $(X_n, d_{X_n}, \nu_{X_n})_{n \ge 1}$ in $\mathfrak{K}$. Lemma~\ref{le:rho} implies that for some sequence $t_n>0$ with $t_n = o(1)$
	\[
	\rho_{\epsilon}^{X_n} \le \rho_{\epsilon + 2 t_n}^{X} + t_n.
	\]
	As $\rho_{(\cdot)}^{X}$ is right-continuous, it follows that 
	\[
	\rho_{\epsilon}^{X_n} \le\rho_{\epsilon}^X + o(1).
	\]
	Hence $\rho_{\epsilon}^{(\cdot)}$ is upper semi-continuous.
\end{proof}

We say a Borel probability measure has \emph{full support}, if any open non-empty set has positive measure. Lemma~\ref{le:rho} and the fact that we only consider compact spaces readily yield the following characterization for GHP-limits.

\begin{corollary}
	\label{co:full}
	Let $(X,d_X, \nu_X)$ be the Gromov--Hausdorff--Prokhorov limit of a sequence $(X_n, d_{X_n}, \nu_{X_n})_{n \ge 1}$.  Then the following statements are equivalent:
	\begin{enumerate}
		\item $\nu_X$ has full support.
		\item $\rho_\epsilon^X>0$ for all $\epsilon>0$.
		\item $\liminf_{n \to \infty} \rho_\epsilon^{X_n} >0$ for all $\epsilon>0$. 
	\end{enumerate}
\end{corollary}
Compare with ~\cite[Lem. 15]{MR2571957}.
\begin{proof}
	The equivalence of the first two statements follows from the compactness of $X$. The equivalence of the second and third statement follow from the inequality
	\[
	\rho_\epsilon^{X_n} \le \rho_{2 \epsilon}^X + o(1) \le \rho_{3 \epsilon}^{X_n} + o(1)
	\]
	which holds for each $\epsilon>0$ by Lemma~\ref{le:rho}.
\end{proof}

We require the following version for random elements of~$\mathfrak{K}$.

\begin{corollary}
	\label{co:concon}
	Let $(\mX,d_{\mX}, \nu_{\mX})$ and $(\mX_n, d_{\mX_n}, \nu_{\mX_n})_{n \ge 1}$ be random elements of $\mathfrak{K}$ satisfying
	\[
	(\mX_n, d_{\mX_n}, \nu_{\mX_n}) \convdis (\mX,d_{\mX}, \nu_{\mX}).
	\]
	Then the following statements are equivalent:
	\begin{enumerate}
		\item $\nu_\mX$ has almost surely full support.
		\item For all $\epsilon, \epsilon'>0$ there are $\delta, N>0$ such that for all $n \ge N$: $\Pr{\rho_\epsilon^{\mX_n} < \delta} < \epsilon'$.
	\end{enumerate}
\end{corollary}
\begin{proof}
	By Corollary~\ref{co:uppercont} it holds for all $\epsilon>0$ that  the map $\rho_{\epsilon}^{(\cdot)}$  is upper semi-continuous with respect to the Gromov--Hausdorff--Prokhorov metric. Thus, for any $\delta>0$, the subset $ \{ (Z, d_Z, \nu_Z) \in \mathfrak{K} \mid \rho_\epsilon^Z < \delta\}$ is open. It follows by the Portmanteau theorem that
	\begin{align}
		\label{eq:doh}
		\liminf_{n \to \infty} \Pr{ \rho_\epsilon^{\mX_n} < \delta} \ge \Pr{ \rho_\epsilon^{\mX} < \delta} .
	\end{align}
	Now suppose that $\nu_{\mX}$ is not fully supported. Then there is an $\epsilon'>0$ with $\Pr{\rho_\epsilon^\mX < \delta} \ge 2\epsilon'$ for all $\delta>0$. It follows by \eqref{eq:doh}  $\Pr{ \rho_\epsilon^{\mX_n} < \delta} \ge \epsilon'$ for all sufficiently large $n$. As this holds for each $\delta>0$, we have shown by contraposition that Claim 2 implies Claim 1. 
	
	It remains to show that Claim 1 implies Claim 2. That is, we assume now that $\nu_{\mX}$ is fully supported. By Skorokhod's representation theorem we may assume that $\mX_n$ converges almost surely towards $\mX$. It follows from Corollary~\ref{co:full}  that  \[
	\liminf_{n \to \infty} \rho_\epsilon^{\mX_n} >0
	\] holds almost surely for each $\epsilon>0$. It follows that for any  $\epsilon'>0$ we may select $\delta>0$ such that
	\[
	\Pr{ \text{there is } N \ge 1 \text{ such that for all } n \ge N:  \rho_\epsilon^{\mX_n} > \delta } > 1 - \epsilon'.
	\]
	Writing this event as a union of increasing events indexed by $N$, it follows that there exists a (deterministic) $N \ge 1$ such that for all $n \ge N$
	\[
	\Pr{\rho_\epsilon^{\mX_n} > \delta} > 1 - \epsilon'.
	\]
\end{proof}

\section{Approximation by random finite subsets}

If the measure $\nu_{\mX}$ of a random measured metric space $(\mX, d_{\mX}, \nu_{\mX})$ has almost surely full support, then that space may be approximated by finite collections points sampled independently according to $\nu_{\mX}$.  
We stress that we generate these subsets in a \emph{quenched} way, that is they are independent copies with the respect to the same instance of the random measure $\nu_{\mX}$.

\begin{lemma}
	\label{le:approx}
	Let $(\mX, d_{\mX}, \nu_{\mX})$ be a random element of $\mathfrak{K}$ such that $\nu_{\mX}$ almost surely has full support. Let $\mY_n \subset \mX$ be formed by collecting the first $n \ge 1$ samples in an infinite sequence of independent $\mu_\mX$-distributed points. We let $\nu_{\mY_n}$ denote the uniform measure on $\mY_n$. Then 
	\begin{align}
		\label{eq:approx}
		\max(d_{\mathrm{H}}(\mY_n, \mX), d_{\mathrm{P}}(\nu_{\mY_n}, \nu_{\mX})) \convas 0.
	\end{align}
\end{lemma}	
\begin{proof}
	We first treat the deterministic case. Let $(X, d_X, \nu_X)$ be a deterministic compact metric space with a Borel probability measure that has full support. For each $n \ge 1$ let $Y_n \subset X$ be formed by taking $n$ samples in an infinite sequence of independent samples of $\nu_X$. As $X$ is compact, it has for each $\epsilon>0$ a finite cover by $\epsilon$-balls $C_1, \ldots, C_k$ for some $k \ge 1$. As $\nu_X$ has full support, it holds that \[ 
	\delta := \min_{1 \le i \le k} \nu_X(C_i) >0.
	\]
	Hence, with probability one there exists an index $N$ such that for all $n \ge N$
	\[
	Y_n \cap C_i \ne \emptyset \quad \text{for all $1 \le i \le k$}.
	\]
	In particular, the Hausdorff distance $d_{\mathrm{H}}$ satisfies for $n \ge N$
	\[
	d_{\mathrm{H}}(Y_n, X) < \epsilon.
	\]
	Thus,
	\[
	d_{\mathrm{H}}(Y_n, X) \convas 0.
	\]
	
	Moreover, the push-forward $\nu_n$ of $\nu_{Y_n}$ along the inclusion $Y_n \subset X$ is the empirical measure corresponding to $\nu_X$. Hence, by a classical result of~\cite[Thm. 3]{MR94839}, it holds with probability $1$ that
	\[
	\nu_n \Longrightarrow \nu_X.
	\]
	Since $X$ is compact and hence separable, it follows that
	\[
	d_{\mathrm{P}}(\nu_n, \nu_X) \convas 0
	\]
	for the Prokhorov distance $d_{\mathrm{P}}$. 
	
	This verifies that when $(X, d_X, \nu_X)$ is a deterministic compact measured metric space with $\nu_X$ having full support, then
	\begin{align}
		\label{eq:approxdetermin}
		\max(d_{\mathrm{H}}(Y_n, X), d_{\mathrm{P}}(\nu_{Y_n}, \nu_X)) \convas 0.
	\end{align}
	
	We assumed that $(\mX, d_{\mX}, \nu_{\mX})$ is a random compact measured metric space such that $\nu_{\mX}$ has almost surely full support. Hence Equation~\eqref{eq:approx} follows from~\eqref{eq:approxdetermin}.
\end{proof}

\end{document}